\documentclass[11pt,a4paper]{amsart}
\usepackage[a4paper, total={15cm, 23cm},centering]{geometry}
\usepackage{amsmath,amssymb,amscd,amsthm,amsxtra,array,soul}
\usepackage[usenames,dvipsnames]{xcolor}
\usepackage[
colorlinks=true,
linkcolor=blue,
citecolor=blue,
urlcolor=blue,
]{hyperref}

\newtheorem*{theorem*}{Theorem}

\newtheorem{theorem}{Theorem}[section]
\newtheorem{lemma}[theorem]{Lemma}
\newtheorem{proposition}[theorem]{Proposition}
\newtheorem{corollary}[theorem]{Corollary}

\theoremstyle{definition}

\theoremstyle{remark}
\newtheorem{remark}[theorem]{Remark}
\numberwithin{equation}{section}
\DeclareMathOperator*{\esssup}{ess\,sup}

\def\R{{\mathbb R}}

\newcommand{\BMO}[0]{\operatorname{BMO}}
\newcommand{\abs}[1]{|#1|}

\newcommand{\la}{\lambda}
\newcommand{\eps}{\varepsilon}
\newcommand{\Z}{\mathbb Z}
\newcommand{\osc}{\operatorname{osc}}

\newcommand{\strt}[1]{\rule{0pt}{#1}}
\newcommand{\car}[1]{\chi_{\strt{1.5ex}#1}}

\newcommand{\norm}[1]{\mbox{$\left\| #1 \right\|$}}

\def\Xint#1{\mathchoice
  {\XXint\displaystyle\textstyle{#1}}%
  {\XXint\textstyle\scriptstyle{#1}}%
  {\XXint\scriptstyle\scriptscriptstyle{#1}}%
  {\XXint\scriptscriptstyle\scriptscriptstyle{#1}}%
  \!\int}
\def\XXint#1#2#3{{\setbox0=\hbox{$#1{#2#3}{\int}$}
    \vcenter{\hbox{$#2#3$}}\kern-.5\wd0}}

\def\avgint{\Xint-}
\newcommand{\vertiii}[1]{{\left\vert\kern-0.25ex\left\vert\kern-0.25ex\left\vert #1 
    \right\vert\kern-0.25ex\right\vert\kern-0.25ex\right\vert}}

\numberwithin{equation}{section}

\allowdisplaybreaks

\def\Item$#1${\item $\displaystyle#1$\hfill}

\begin{document}

\title[Extensions of John--Nirenberg and applications]{Extensions of the John--Nirenberg theorem and applications}

\author[J. Canto]{Javier Canto}
\address[Javier Canto] {BCAM \textendash  Basque Center for Applied Mathematics, Bilbao, Spain}
\email{jcanto@bcamath.org}

\author[C. P\'erez]{Carlos P\'erez}
\address[Carlos P\'erez]{ Department of Mathematics, University of the Basque Country, IKERBASQUE 
(Basque Foundation for Science) and
BCAM \textendash  Basque Center for Applied Mathematics, Bilbao, Spain}
\email{cperez@bcamath.org}

\thanks{Both authors are supported by the Spanish Ministry of Economy and Competitiveness, MTM2017-82160-C2-2-P and SEV-2017-0718 and by the Basque Government BERC 2018-2021 and IT1247-19.
J.C. is also supported by Basque Government through ``Ayuda para la formaci\'on de personal investigador no doctor".}

%\subjclass{Primary: 42B25. Secondary: 43A85.}

%\keywords{Muckenhoupt weights, BMO}

\begin{abstract}
The John--Nirenberg theorem states that functions of bounded mean oscillation are exponentially integrable. In this article we give two extensions of this theorem. The first one relates the dyadic maximal function to the sharp maximal function of Fefferman--Stein, while the second one concerns local weighted mean oscillations, generalizing a result of Muckenhoupt and Wheeden.   Applications to the context of generalized Poincar\'e type inequalities and to the context of the $C_p$ class of weights are given. Extensions to the case of polynomial $\BMO$ type spaces are also given.
\end{abstract}

\maketitle

% \tableofcontents

\section{Introduction and main results}

The classical John--Nirenberg (JN) theorem \cite{JN} states that any function of bounded mean oscillation is locally exponentially integrable, see for example \cite{GCRdF}. More precisely, we have the inequality
\begin{equation}\label{True-John-Nirenberg}
|\{ x\in Q: |f(x)-f_Q| > t \}| \leq c_1 e^{-c_2 t / \| f\|_{\BMO}}|Q|.
\end{equation}

The main purpose of this work is to provide two extensions of the classical JN theorem for $\BMO$ functions. 
The first extension constitutes an improvement of a result of Karagulyan \cite{Karagulyan}, which is in turn a more precise version of the classical Fefferman--Stein inequalities relating the Hardy--Littlewood and the sharp maximal functions.
The second extension is an improvement of some classical estimates by Muckenhoupt and Wheeden \cite{MW} concerning weighted local mean oscillations. These estimates were already discussed in the recent work \cite{OPRRR19} in a more restrictive setting.

	These two extensions, although different a priori, are obtained by a similar method, using some ideas from the recent work \cite{PR-Poincare}.

\subsection{Improving Karagulyan's estimate}

The first extension is motivated by the work of Karagulyan \cite{Karagulyan}, who already provided an extension of the JN theorem. We improve this interesting result by providing a different more flexible proof with several different advantages. 
However,
this first extension is also inspired by the recent work \cite{PR-Poincare}, where a different approach to the main results from \cite{FKS} concerning degenerate Poincar\'e--Sobolev inequalities is found. Using this alternative approach, we give a direct proof of the classical JN theorem (see the Appendix).

We obtain two different consequences of this improvement of the JN theorem. Firstly, we derive some degenerate Poincar\'e--Sobolev endpoint inequalities not available from the methods in \cite{PR-Poincare}.
Secondly, this improvement will be applied within the context of the $C_p$ class of weights (see Section \ref{prelim}) solving some questions left open in \cite{CCF}.

To establish this result we recall the sharp maximal function introduced by C. Fefferman and E. Stein. It is defined by
\begin{equation*}
M^\sharp h(x) = \sup_{x \in R} \frac{1}{|R|}\int_R |h-h_R|,
\end{equation*}
where the supremum is taken over all cubes $R$ that contain the point $x$, and $h_R= \frac{1}{|R|}\int_Rh$ denotes the average of $h$ over $R$.
Karagulyan proved in \cite{Karagulyan} the following interesting exponential decay. If $f$ is an $L^1_{loc}$ function and $B$ a ball in $\R^n$, then
\begin{equation*}
%\label{Karagu}
|\{x\in B: \frac{|f(x)-f_B|}{M^\sharp f(x)}>\la\}| \lesssim e^{-c \, \la} |B|.
\end{equation*}
%
%{\color{blue}
 Our first main result, Theorem \ref{Teorema1},  improves this exponential decay in several ways. On one hand, we have the decay for the local maximal function and on the other hand, we obtain weighted estimates. Our method of proof is different from \cite{Karagulyan}.

In order to state our first main result we set the following notation for $r>1$ and for any cube $Q$
\[ 
w_r(Q)= |Q| \left( \avgint_Q w^r\right)^\frac{1}{r} = |Q|^{\frac{1}{r'}} \left( \int_Q w^r\right)^\frac 1r.
\]

\begin{theorem}\label{Teorema1}
Let $f$ be a locally integrable function. Then for any cube $Q$, for any $1\leq p <\infty$ and $1<r<\infty$, the following estimate holds
\begin{equation}\label{John-Nirenberg1}
\left(  \frac{1}{w_r(Q)} \int_Q \left(\frac{M_Q(f-f_Q)(x)}{M^\sharp f(x)}\, \right)^p w(x)dx\right)^\frac{1}{p} \leq c_n \, pr'.
\end{equation}
Hence, if further $w\in A_{\infty}$ we have 
\begin{equation}\label{John-Nirenberg2}
\left(  \frac{1}{w(Q)} \int_Q \left(\frac{M_Q(f-f_Q)(x)}{M^\sharp f(x)}\, \right)^p w(x)dx\right)^\frac{1}{p} \leq c_n \, p\,[w]_{A_{\infty}}.%\footnote{hay que dar al menos la notacion}
\end{equation}
\end{theorem}

Here, $M_Q$ is the local dyadic maximal operator over $Q$, see Section \ref{PruebaTeorema1} for the precise definition.

In Section \ref{polynomial} we extend this result to the context of polynomial type $\BMO$. As a direct corollary of Theorem \ref{Teorema1} and Proposition \ref{EXPONENTIAL}, we have the following exponential estimates.

\begin{remark}
We remark that the corresponding result replacing the $L^p$ norm by the (larger) Lorentz norm $L^{p,q}$ with $1\leq q<p$ cannot be proved even in the simplest situation $w=1$ and without $M$. 

\end{remark}

\begin{remark}
 We also remark that the factor $p$ in \eqref{John-Nirenberg1} (or \eqref{John-Nirenberg2}) it is crucial since it yields the exponential type result as follows. 

\end{remark}

\begin{corollary}\label{COROLARIO}
Let $f$ be a locally integrable function. Then we have
\begin{itemize}
\item Improved John--Nirenberg estimate:
\[ \norm{\frac{\displaystyle M_Q(f-f_Q)}{\displaystyle M^\sharp f}}_{\exp L(Q, \frac{wdx}{w(Q)})} \leq c_n\,[w]_{A_{\infty}}
\]
meaning that there exist dimensional constants $c_1,c_2>0$ such that
\begin{equation*}%\label{improvedkarag} ç 
w\left(\left\{ x\in Q: \frac{M_Q(f-f_Q)}{M^\sharp f(x) }>t \right\}\right) \leq c_1 \,e^{-c_2\, t / [w]_{A_\infty}}\, w(Q), \quad t>0.
\end{equation*}
\item For every cube and $\la, \gamma >0$ we have the following good-$\la$ type inequality
\begin{equation*}   % \label{good-lambda}
w\big( \{x\in Q: M_Q(f-f_Q)> \la, M^\sharp f(x) \leq \gamma \la \} \big)  \leq c_1\,{e^{\frac {-c_2}{\gamma[w]_{A_\infty}}}}\, w(Q)
\end{equation*} 
\end{itemize}
\end{corollary}

We call this result improved John--Nirenberg estimate because if $w=1$ and $f\in \BMO,$ then $M^\sharp f(x) \leq \|f\|_{\BMO}$ for a.e.  $x$ and, therefore,
$$
\left|\left\{ x\in Q: M_Q(f-f_Q)>t \right\}\right|\leq c_1 \,e^{\frac{-c_2\, t}{\|f\|_{\BMO}}}\, |Q|, \quad t>0.
$$
This implies the JN theorem \eqref{True-John-Nirenberg} by Lebesgue differentiation theorem, because $M_Q(f-f_Q)\geq f-f_Q$ a.e. in $Q.$

{
%\color{blue}

$\bullet$ {\it Generalized Poincar\'e inequalities: sharp cuantitative $A_{\infty}$ bounds.} 
As a first application of Theorem \ref{Teorema1}, we improve the main result in \cite{FPW98} (at least in the simplest situation of cubes) which at the same time provides a limiting result that could not be treated  in Theorem 1.14  of \cite{PR-Poincare}.

Let $w$ be an $A_\infty$ weight and let $a$ be a functional over cubes of $\mathbb R^n$. We will assume that $a$ satisfies the $D_r(w)$ condition for some $r>1$ as introduced in \cite{FPW98}. More precisely, for every cube $Q$ and every collection $\Lambda$ of pairwise disjoint subcubes of $Q$, the following inequality holds:
\begin{equation}\label{DR-condition}
\sum_{P\in \Lambda} w(P) \, a(P)^r \leq \|a\|^r w(Q) a(Q)^r,
\end{equation}
for some constant $\|a\|>0$ that plays the role of the ``norm" of $a$.

These kind of functionals were studied in relation with self improvement properties of generalized Poincar\'e inequalities in 
\cite{FPW98}, further studied in \cite{MacManus-Perez-98} and more recently improved in \cite{PR-Poincare}. We establish now an new endpoint result in the spirit of Theorem 1.14  in  \cite{PR-Poincare} which was missing since Theorem  \ref{Teorema1} was not available. 

\begin{theorem}\label{Automejora}
Let $w\in A_\infty$ and $a$ a functional satisfying \eqref{DR-condition}. Let $f$ be a locally integrable function such that for every cube $Q$,
\begin{equation}\label{starting-point}
\frac{1}{|Q|}\int_Q |f-f_Q| \leq a(Q).
\end{equation} 
Then, for every cube $Q$,
$$
\| f-f_Q\|_{L^{r,\infty}\big( Q, \frac{w}{w(Q)}\big)} \leq c_n \, r\, [w]_{A_\infty}\,\|a\| \, a(Q).
$$
\end{theorem}

\begin{remark}

The method in  \cite{FPW98}, based on the good-$\lambda$ method of Burkholder--Gundy \cite{BG}, yields an exponential bound in  $[w]_{A_\infty}$. We still use here the good-$\lambda$ method but we use instead Corollary \ref{COROLARIO}. %\eqref{John-Nirenberg} or rather its consequence \eqref{improvedkarag}. 

\end{remark}

$\bullet$ {\it The $C_p$ condition.} 
As a second application of Theorem \ref{Teorema1} we provide an improvement of theorem of K. Yabuta \cite{YCF} concerning a classical inequality of Feffereman--Stein relating the Hardy-Littlewood maximal function $M$ and the sharp maximal function $M^\sharp$ introduced by them in \cite{FS}. 

This result of Yabuta considers the $C_p$ class of weights which is a larger class of weights than the $A_{\infty}$ class of Muckenhoupt (see Section \ref{prelim}).

\begin{theorem}[Quantitative norm inequality]\label{Yabuta-quantif}
Let $1<p<q<\infty$ and $w\in C_q$. Then for any $f\in L^\infty_c(\R^n)$ we have
$$ \norm{Mf}_{L^p(w)} \leq  c_n \frac{pq}{q-p} \max(1,[w]_{C_q} \log^+[w]_{C_q}) \|M^\sharp f\|_{L^p(w)},$$
where the constant $c_n$ only depends on $n$.
\end{theorem}

\begin{remark}
We remark that, as a consequence of Corollary \ref{COROLARIO}, we can also obtain the following weighted inequality for $A_\infty$ weights by standard arguments:
$$\norm{Mf}_{L^p(w)} \leq c [w]_{A_\infty} \ \|M^\sharp f\|_{L^p(w)}, \quad 0<p<\infty.$$
This inequality is not new, see for example \cite{LERNER06}.
\end{remark}

Very recently, A. Lerner \cite{Lerner19} proved a characterization of weights satisfying a weak Fefferman--Stein inequality 
\[ \|f\|_{L^{p,\infty}(w)} \leq C \| M^\sharp f\|_{L^p(w)}.\]
The weights satisfying this inequality are of a different class of weights, called $SC_p$ (strong $C_p$). This class is contained in $C_p$ and contains $C_{p+\eps}$ for every $\eps>0$.

Theorem \ref{Yabuta-quantif} has a straight application to the wide class of operators described in \cite{CLPRCF}. Indeed, we say that an operator satisfies the \eqref{PROPIEDAD} property if there  are some constants $\delta\in(0,1)$ and $c>0$ such that for all $f$,
\begin{equation}
\label{PROPIEDAD}\tag{D} M^\sharp_\delta (Tf)(x) \leq c Mf(x), \quad a.e. \: x.
\end{equation}
Here $M$ denotes the standard Hardy--Littlewood maximal operator and $M^\sharp _\delta f = M^\sharp(f^\delta)^\frac1	\delta$.  This property is modeled by a result in \cite{AP} where \eqref{PROPIEDAD} was proved for any Calder\'on--Zygmund operator. It also holds for some square function operators and some pseudo-differential operators. The multilinear Calder\'on--Zygmund operators version was obtained in \cite{LOPTT}. There is a more exhaustive list in \cite{CLPRCF}.

%\azul{
\begin{corollary}\label{CpOperator(D)condition}
Let $1<p<q<\infty$ and $T$ be an operator that satisfies the property \eqref{PROPIEDAD} with constant $C_T$ for some $\frac pq <\delta<1$. Then for $w\in C_q$ we have
$$
\norm{Tf}_{L^p(w)} \leq c_n\ C_T \ \left(\frac{pq}{\delta q-p} \  \max(1,[w]_{C_q}\log^+[w]_{C_q}) \right)^\frac1\delta \norm{Mf}_{L^p(w)}.
$$
\end{corollary}

\subsection{Weighted mean oscillation}

The second extension of the JN theorem we consider in this paper is motivated by the following classical result of  Muckenhoupt and Wheeden in \cite{MW}. 

Following the language in \cite{MW}, a function $f$ is said to be bounded mean oscillation with weight $w$ if there exists some $C>0$ such that for every cube $Q$,
\[
\int_Q |f-f_Q| \leq C \ w(Q).
\]
This class of functions is interesting because it is connected to the theory of weighted Hardy spaces \cite{GC} 
and to the context of extrapolation \cite{HMS} (see more details and a new proof in \cite{CPR}).  

\begin{theorem*}[Muckenhoupt--Wheeden]
Let $1\leq p<\infty$ and $w\in A_p$. Then $f$ is of bounded mean oscillation with weight $w$ if and only if  for every $1\leq r<\infty$ satisfying $1\leq r\leq p'$, there exists a constant such that, for all cubes $Q$,
\begin{equation}\label{Muckenhoupt-Wheeden-original}
\int_Q |f-f_Q|^r w^{1-r} \leq c\, w(Q).
\end{equation}
\end{theorem*}

As was shown in \cite{MW}, the range $1\leq r\leq p'$ is optimal, since for any given $p>1$ there exist $f,w$ for which $w\in A_q$ for all $q>p$ but \eqref{Muckenhoupt-Wheeden-original} fails for $r=p'$. 

In \cite{OPRRR19} the authors obtained a mixed-type $A_p$--$A_\infty$ quantitative estimate of inequality \eqref{Muckenhoupt-Wheeden-original}. Here we are going to improve Theorem 1.7 from that paper, using a simplified and more transparent argument that avoids completely the use of sparse domination.

The main idea is closely related to the proof of Theorem \ref{Teorema1} above.

For a weight $w$ and $p>1, r\geq1$, we define the following bumped $A_p$ constant %\footnote{cambié un poco la notacion pero no estoy seguro}
$$
[w]_{A^r_p} = \sup_Q \left( \avgint_Q w^r\right)^\frac1r \left( \avgint_Q w^{1-p'}\right)^{p-1}.
$$
Note that $[w]_{A_p}\leq [w]_{A^r_p}$ for $r\geq1$.

\begin{theorem}\label{bumped-muckenhoupt-wheeden} Let $p,r>1$, $w$ such that $[w]_{A_p^r}<	\infty$ and let $f$ locally integrable such that 
\begin{equation*}%\label{bumped-starting-point}
\| f \|_{\BMO_{w,r}}:= \sup_Q \frac{1}{w_r(Q)} \int_Q |f-f_Q|    <\infty.
\end{equation*}
Then we have the estimate 
\begin{equation*}
\left(\frac{1}{w_r(Q)}\int_Q \left( \frac{|f(x)-f_Q|}{w(x)}\right)^{p'} w(x)dx\right)^\frac{1}{p'} \leq c_n p' [w]_{A^r_p}^\frac1p \,(r')^\frac{1}{p'} \|f\|_{\BMO_{w,r}} .
\end{equation*}
\end{theorem}

}

\begin{corollary}\label{ThmBloomBMO}
Let $f\in \BMO_{w,1}$, namely $\sup_Q\frac{1}{w(Q)}\int_Q\abs{f-f_{Q} }<\infty$. 
\begin{enumerate}
\item If $w\in A_{1}$ we have that for every $q>1$, 
\[
\left(\frac{1}{w(Q)}\int_{Q}\left|\frac{f(x)-f_{Q}}{w(x)}\right|^{q}w(x)dx\right)^{\frac{1}{q}}\leq c_{n}\, q \, [w]_{A_{1}}^{\frac{1}{q'}}\, [w]_{A_{\infty}}^{\frac{1}{q}}\, \|f\|_{\BMO_{w,1}},
\]
and hence for any cube $Q$
\begin{equation} \label{genJN}
\norm{ \frac{ f- f_{Q} }{w} }_{ \exp L\big(Q,\frac{w(x)dx}{w(Q)} \big) } \, \leq c_n\, [w]_{A_1}\, \|f\|_{\BMO_{w,1}}.
\end{equation}
\item If $w\in A_{p}$ with $1<p<\infty$ then, 
\[
\left(\frac{1}{w(Q)}\int_{Q}\left|\frac{f(x)-f_{Q}}{w(x)}\right|^{p'}w(x)dx\right)^{\frac{1}{p'}}\leq c_{n}\, p'\, [w]_{A_{p}}^{\frac{1}{p}}\,[w]_{A_{\infty}}^{\frac{1}{p'}}\,\|f\|_{\BMO_{w,1}} .
\]
\end{enumerate}
\end{corollary}

% We will prove this theorem in Section \ref{PruebahmBloomBMO}. 

\begin{remark}
{It follows from \eqref{genJN} that for} any $t>0$ then
\[
w(\{ x\in Q : |f(x)-f_Q|>t \, w(x)  \}) \leq 2 e^{-\frac{\scriptstyle t}{\scriptstyle c_n [w]_{A_1} \|f\|_{\BMO_{w,1}}}} \, w(Q).
\]
\end{remark}

We extend this result to the context of  polynomial type $\BMO$ in Section \ref{polynomial}. It is not clear how to obtain these  new  polynomial $\BMO$ type estimates from the sparse method used in \cite{OPRRR19}.

\section{Some preliminaries and notation}
\label{prelim}

\subsection{Weighted inequalities}

A weight is a non-negative locally integrable function $w$. For $1\leq p<\infty$, we say that $w\in A_p$ if and only if $[w]_{A_p}<\infty$, where
\[
[w]_{A_1} = \esssup_{x\in \R^n} \frac{Mw(x)}{w(x)},
\]
\[ [w]_{A_p} = \sup_Q \left( \avgint_Q w\right)\left(  \avgint_Q w^{1-p'}\right)^{p-1}, \quad 1<p<\infty .
\]
The $A_\infty$ class is defined as the union of all other $A_p$ classes, that is,
\[ A_\infty = \bigcup_{1\leq p<\infty} A_p.\]
We define the $A_\infty$ constant as
\[ [w]_{A_\infty} = \sup_Q \frac{\int_Q M(w\car Q)}{w(Q)}.\]
We state the sharp reverse H\"older inequality for $A_\infty $ weights obtained in \cite{HP} (see also \cite{HPR}).

\begin{theorem}[Sharp RHI, \cite{HP}]\label{SHARP-RHI-Ainfty} Let $w\in A_\infty$ and $\delta = \frac{1}{c_n[w]_{A_\infty}}$. Then for every cube $Q$, we have
\[
 \left( \avgint_Q w^{1+\delta} \right)^\frac{1}{1+\delta} \leq 2\ \avgint_Q w.
\]
\end{theorem}

For $1<p<\infty$, the $C_p$ class of weights is defined as the weights $w$ for which there exist $C,\eps$ such that for every cube $Q$ and $E\subset Q$, 
\[ w(E) \leq C \left( \frac{|E|}{|Q|}\right)^\eps \int_{\R^n} (M\car Q)^pw. \]
The $C_p$ constant was defined in \cite{CCF} as
\[ [w]_{C_p} = \sup_Q \frac{\int_Q M(w\car Q)}{\int_{\R^n} (M\car Q)^p w}.\]
These weights also satisfy a sharp RHI as found in \cite{CCF}:
\begin{theorem}[Sharp RHI, \cite{CCF}]\label{SHARP-RHI-Cp} Let $1<p<\infty$, $w\in C_p$ and $\delta = \frac{1}{c_{n,p}([w]_{C_p}+1)}$. Then for every cube $Q$, we have
\[
 \left( \avgint_Q w^{1+\delta} \right)^\frac{1}{1+\delta} \leq \frac{2}{|Q|} \int_{\R^n} (M\car Q)^pw.
\]
\end{theorem}

\subsection{Exponential estimates}

\begin{proposition}\label{EXPONENTIAL}
Suppose that $(X,\mu)$ is a probability space and $f$ a non-negative function such that for every $1\leq p<\infty$ we have the $L^p$ bound
\begin{equation*}
\left( \int_X f(x)^p d\mu \right)^\frac{1}{p} \leq \gamma \, p,
\end{equation*}
for some constant $\gamma$ independent from $p$. Then $f\in \exp(L)(X,\mu)$, meaning 
$$\mu( \{x\in X : f(x)>t\}) \leq  e^{-\frac{\scriptstyle t}{\scriptstyle 4 \gamma}}, \quad t>0.$$ 
\end{proposition}

\begin{proof}
We compute
\begin{align*}
\int_X \left(\exp  \frac{f(x)}{4\gamma}  -1 \right)d\mu &
 = \sum_{n=1}^\infty \frac{1}{n!} \int_X \left(\frac{f(x)}{4\gamma}\right)^n d\mu 
 \leq \sum_{n=1}^\infty \frac{1}{n!}\left(\frac{n}{4}\right)^n \leq 1.
\end{align*}
Therefore,
\begin{align*}
\mu( \{x\in X: f(x) >t\}) & 
= \mu( \{x \in X : \frac{f(x)}{4\gamma} - \frac{t}{4\gamma}-\log 2>\log2\}) \\ 
& \leq \int_X \left( \exp\big( \frac{f(x)}{3\gamma}-\frac{t}{4\gamma}-\log 2\big) -1\right) d\mu \\
& = 2 e^{-\frac{ \scriptstyle t}{\scriptstyle 4\gamma}}\int_X \left( \exp \frac{f(x)}{4\gamma} -1\right) d\mu. \qedhere
\end{align*}
\end{proof}

Here we present a minimization lemma that we will use in the proofs of Theorems \ref{Teorema1} and \ref{ThmBloomBMO}, as well as Proposition \ref{proposicion-JN}.
\begin{lemma}\label{minimization}
Let $0<\alpha <\infty$. Then
$$
\min_{1<t<\infty}\, t\, \frac{ t^\alpha}{t^\alpha-1} \leq e \left(1+\frac1\alpha \right).
$$
\end{lemma}
%
%\begin{proof}
%The function $\varphi(t) = t^{\alpha+1}(t^\alpha-1)^{-1}$ tends to infinity at 1 and infinity. So, if the derivative vanishes at a unique point, that point has to be a global minimum. The derivative has the expression
%%
%$$\varphi'(t) = \frac{(\alpha+1) t^\alpha (t^\alpha-1) - \alpha t^{2\alpha}}{(t^\alpha-1)^2},$$
%%
%which vanishes only at $t= (\alpha+1)^\frac{1}{\alpha}$. Therefore, the global minimum is
%%
%\begin{align*}
%\varphi\big( (\alpha+1)^\frac{1}{\alpha} \big) & = (\alpha+1)^\frac{1}{\alpha} \, \frac{\alpha+1}{\alpha}\leq  e \left(1+\frac1\alpha \right). \qedhere
%\end{align*}
%%
%\end{proof}

\section{Karagulyan's result revisited}\label{PruebaTeorema1}

 In this section we give a new inequality that can be seen as an extension of the JN theorem that requires minimal hypothesis on the function. This new inequality concerns the dyadic maximal operator and the sharp maximal operator. Given a cube $Q$ we define $M_Q$, the \emph{localized dyadic maximal operator}, acting on a function $h$ by
\begin{equation*}
M_Qh(x) = \sup_{\substack{R\in \mathcal D(Q)\\x\in R}} \frac{1}{|R|}\int_R |h|.
\end{equation*}
Here $\mathcal D(Q)$ denotes the collection of all dyadic descendants of $Q$. We don't need that $Q$ belongs to any particular dyadic family, even though the supremum is taken over the dyadic collection generated by $Q$.

\begin{proof}[Proof of Theorem \ref{Teorema1}] The main idea for the proof comes from \cite{PR-Poincare}. 
Fix a cube $Q$. We make the local Calder\'on--Zygmund decomposition in the cube $Q$ of the function \[ F(x)= \frac{|f(x)-f_Q|}{\text{osc}(f,Q)}\] at height $\la>1$ to be precised later. We have used the notation 
$$\text{osc}(f,Q) = \avgint_Q|f-f_Q|.$$ 
More precisely, we choose the dyadic subcubes $\{Q_j\}$ of $Q$, maximal for the inclusion among the cubes $R$  that satisfy 
$\avgint_R F(x)dx >\la$. The cubes $\{Q_j\}$ are pairwise disjoint and satisfy the following properties:
\begin{itemize}
\item  $ \text{osc}(f,Q) \la < \avgint_{Q_j} {|f-f_Q|} \leq 2^n \la\ \text{osc}(f,Q) $,
\item  $ \sum_j |Q_j| \leq \frac{|Q|}{\la}$,
\item  For $  x\not \in \cup_j Q_j,\quad  \displaystyle {M_Q(f-f_Q)(x)}\leq \la\ {\text{osc}(f,Q)} .$
\end{itemize}
The first two properties follow from the stopping time and the maximality. To prove the third one, note that $\avgint_R |f-f_Q| / \text{osc}(f,Q) \leq \la$ for all dyadic $R$ that contains $x$.

Now, by maximality of the cubes, for $x\in Q_j$ we can localize the maximal function in the following way
\begin{equation}\label{localizador}{M_Q(f-f_Q)(x)}  = M_{Q_j}(f-f_Q)(x) \leq M_{Q_j}(f-f_{Q_j})(x)+|f_Q-f_{Q_j}|.
\end{equation}
Moreover, for $x \in Q_j$, we have
\begin{equation}\label{C-Zlocalizado}
\frac{|f_Q-f_{Q_j}|}{M^\sharp f(x) } = \frac{| \avgint_{Q_j}(f-f_Q)|}{M^\sharp f(x)} \leq \frac{\avgint_{Q_j} |f-f_Q|}{\text{osc}(f,Q)} \leq 2^n \la,
\end{equation}
by the Calder\'on--Zygmund decomposition. Thus, we have found the following pointwise bound, for a.e. $x\in Q,$
\begin{align*}
\frac{M_Q(f-f_Q) (x)}{M^\sharp f(x)} &  = \frac{M_Q(f-f_Q) (x)}{M^\sharp f(x)}\car{Q\setminus \cup_j Q_j}(x) + \sum_j \frac{M_{Q}(f-f_{Q}) (x)}{M^\sharp f(x)} \car{Q_j}(x)  \\
& \leq \la \car{Q\setminus \cup_j Q_j}(x) +\sum_j \left( \frac{M_{Q_j}(f-f_{Q_j}) (x)}{M^\sharp f(x)} + \frac{|f_Q-f_{Q_j}|}{M^\sharp f(x) } \right) \car{Q_j}(x) \\
& \leq \la \car{Q\setminus \cup_j Q_j}(x) +\sum_j \left( \frac{M_{Q_j}(f-f_{Q_j}) (x)}{M^\sharp f(x)} + 2^n \la\right) \car{Q_j}(x) \\
& \leq  2^n \la + \sum_j \frac{M_{Q_j}(f-f_{Q_j}) (x)}{M^\sharp f(x)}  \car{Q_j}(x).
\end{align*}
We have used \eqref{localizador} and \eqref{C-Zlocalizado} in the first  and second inequalities respectively.

Now we compute the norm. Using the triangular inequality, Jensen's inequality and the fact that the $Q_j$ are pairwise disjoint, we get
\begin{align*}
 \Big( \frac{1}{w_r(Q)}\int_Q & \Big( \frac{M_Q(f-f_Q)}{M^\sharp f} \Big)^p\,  w(x)dx \Big)^{\frac1p} \\
 & \leq 2^n \la + \left( \sum_j \frac{w_r(Q_j)}{w_r(Q)} \frac{1}{w_r(Q_j)}\int_Q \left( \frac{M_{Q_j}(f-f_{Q_j})}{M^\sharp f(x)} \right)^p \,w(x)dx\right)^\frac1p \\
& \leq 2^n\la +X \left(  \frac{1}{w_r(Q)}\sum_j w_r(Q_j) \right)^\frac1p \\
& \leq 2^n \la +\frac{X}{\la^{\frac{1}{pr'}}}. 
\end{align*}
We have used that, by H\"older's and one of the main properties of the family $Q_j$,
\begin{equation*}
\sum_{j} w_r(Q_j) 
\leq   w_r(Q)  \left (\frac{1}{\lambda}\right )^{\frac{1}{r'}},
\end{equation*}
and we have set  
\[ X =   \sup_{R \in \mathcal D} \left( \frac{1}{w_r(R)} \left( \frac{M_R(f-f_R)}{M^\sharp f}\right)^p \,w(x)dx\right)^{\frac1p}. \]

Now take the supremum over all dyadic cubes $Q$ and obtain, for arbitrary $\la>1$,
$$X \leq 2^{n} \la +\frac{X}{\la^{\frac{1}{pr'}}}. $$
This in turn implies, if we assume $X<\infty$, that
$$X\leq 2^n \la \frac{ \la^{\frac{1}{pr'}} }{\la^{\frac{1}{pr'}}-1} .$$
Applying  Lemma \ref{minimization}, we see  $X\leq c_n \, pr'$, since $\lambda >1$ was free.
This finishes the proof if we assume that $X<\infty$.

In order to remove the hypothesis $X<\infty$, it is enough to work with 
$$X_{\eps,K} := \sup_{Q\in \mathcal D} \left( \frac{1}{w_r(Q)}\int_Q \left( \frac{M_Q(f_K-(f_K)_Q)}{M^\sharp (f_K)+\eps}\right)^p wdx\right)^{\frac1p}\leq 2 \frac K \eps <\infty$$
for a suitable truncation $f_K$ of $f$ at height $K$. For example, one can take
\begin{equation*}%\label{truncation}
f_K(x) = \begin{cases}
-K, & f(x) < -K, \\
f(x), & -K \leq f(x) \leq K,\\
K, & K < f(x).
\end{cases}
\end{equation*}
Making the same computations as above with some trivial changes, we can obtain the bounds for $X_{\eps,K}$ independently of $\eps$ and $K$. Finally, monotone convergence finishes the argument.
\end{proof}

\begin{remark} \label{REMARK-DIADICO}
Since throughout the proof the only cubes that appear are dyadic descendants of $Q$, we actually obtain the stronger estimate
\[
\left( \frac{1}{w_r(Q)}\int_Q \left(\frac{M_Q(f-f_Q)(x)}{M_Q^\sharp f(x)}\, \right)^p w(x)dx\right)^\frac{1}{p} \leq c_n \, pr',
\]
where $M_Q^\sharp$ is the sharp operator taking the supremum over dyadic descendants of $Q$. Since $M_Q^\sharp \leq M^\sharp,$ this last estimate is stronger.
\end{remark}

Finally, if $w\in A_\infty,$ we choose $r = 1+\delta$ with $\delta$ as in Theorem \ref{SHARP-RHI-Ainfty}. This way, $w_r(Q) \leq 2 w(Q)$ and $r'\lesssim [w]_{A_\infty}$. From this observation, the $A_\infty$ estimate in Theorem \ref{Teorema1} follows.

\section{Weighted local mean oscillation}

Here we give the proof of Theorem \ref{bumped-muckenhoupt-wheeden} involving {bump} conditions on the weight. For a weight $w$, a cube $Q$ and $r>1$, we defined the bump functional
$$w_r(Q)= |Q| \left( \avgint_Q w^r\right)^\frac{1}{r} = |Q|^{\frac{1}{r'}} \left( \int_Q w^r\right)^\frac 1r,$$
and the bumped $A_p$ constant
$$
[w]_{A^r_p} = \sup_Q \left( \avgint_Q w^r\right)^\frac1r \left( \avgint_Q w^{1-p'}\right)^{p-1}.
$$

\begin{proof}[Proof of Theorem \ref{bumped-muckenhoupt-wheeden}]

For a fixed a cube $Q$, we have to prove
\begin{equation}
\left(\frac{1}{w_r(Q)}\int_Q \left( \frac{|f(x)-f_Q|}{w(x)}\right)^{p'} w(x)dx\right)^\frac{1}{p'} \leq c_n p' [w]_{A^r_p}^\frac1p \,(r')^\frac{1}{p'} \|f\|_{\BMO_{w,r}} .
\end{equation}
We recall that 
$$\| f \|_{\BMO_{w,r}}:= \sup_Q \frac{1}{w_r(Q)} \int_Q |f-f_Q|. 
$$
We may suppose by homogeneity that $\|f\|_{\BMO_{w,r}}=1$. Hence, if we let $L>1$, to be chosen later, we can choose a family of maximal subcubes $\{Q_j\}$ in $Q$ such that 
\begin{equation}\label{starting-point-bumped}
\frac{1}{w_r(Q_j)} \int_{Q_j} |f-f_Q| > L.
\end{equation}
Observe that if the family is empty we can see that \, $|f(x)-f_Q| \leq Lw(x)$\, a. e. $x\in Q$ and the result is trivial. Also since $\|f\|_{\BMO_{w,r}}=1$, we have that $Q$ is not one of the selected cubes. We can check that, if $Q_j'$ denotes the ancestor of $Q_j$, the following properties hold: 

\begin{itemize}
\item $\displaystyle \frac{1}{w_r(Q'_j)} \int_{Q'_j} |f-f_Q|\,dx\leq L $,
\item $\displaystyle |f_{Q_j}-f_{Q}|\leq 2^n\,L\, \left( \avgint_{Q_j'}w^r\right)^\frac1r $,\\
\item $\displaystyle \sum_j w_r(Q_j) \leq \frac{w_r (Q)}{L}$ \quad by \eqref{starting-point-bumped} and since since $\|f\|_{BMO_{w,r}}=1$,
\item $\displaystyle |f(x)-f_Q| \leq Lw(x)$ \ for a.e. $x\not \in \cup_j Q_j$.\\
\end{itemize} 

Using the disjointness, we have for $a.e. \ x\in Q$,
\begin{align*}
f(x)-f_Q & = (f(x)-f_Q) \car {(\cup_j Q_j)^c} (x) + \sum_{j} (f_{Q_j}-f_Q) \car{Q_j}(x) + \sum_j (f(x)-f_{Q_j})\car{Q_j}(x) \\
&= A_1(x) + A_2(x) + B(x).
\end{align*}

Since $p'> 1$ we can use the triangular inequality to get
\begin{align*}
\left( \frac{1}{ w_r(Q) } \int_{ Q }   \left(\frac{|f -f_{Q}|}{w}\right)^{p'}     \,wdx\right)^{\frac{1}{ p'} } 
& \leq \left( \frac{1}{w_r(Q)}\int_{(\cup_j Q_j)^c} \left(\frac{|f-f_Q|}{w}\right)^{p'}w \right)^\frac{1}{p'} \\
& \quad + \left( \frac{1}{w_r(Q)}\sum_j \int_{Q_j} \left(\frac{|f_{Q_j}-f_Q|}{w}\right)^{p'}w \right)^\frac{1}{p'} \\
& \quad +\left( \frac{1}{w_r(Q)}\sum_j\int_{Q_j} \left(\frac{|f-f_{Q_j}|}{w}\right)^{p'}w \right)^\frac{1}{p'} \\
& =  A_1+A_2+B.
\end{align*}

Now, since $w(Q)\leq w_r(Q)$ the first term is $A_1\leq L$. To bound $B$  we denote
$$X = \sup_R \left( \frac{1}{w_r(R)}\int_R \left( \frac{|f-f_R|}{w}\right)^{p'} w\right)^\frac{1}{p'}.$$
and use that $\sum_j w_r(Q_j) \leq \frac{w_r (Q)}{L}$,   the third property of the family of the cubes $\{Q_j\}$, to obtain: 
\begin{align*}
B & \leq  X \left( \frac{1}{w_r(Q)}\sum_j w_r(Q_j)\right)^\frac{1}{p'}
\leq X \left( \frac1L \right)^\frac{1}{p'}.
\end{align*}

The argument for bounding $A_2$ is more delicate. We start the computations:
\begin{align*}
A_2 & = \left(\frac{1}{w_r(Q)} \sum_j\int_{Q_j}|f_{Q_j }-f_Q|^{p'}w^{p'-1}\right)^\frac{1}{p'} \\
&\leq 2^n L \left( \frac{1}{w_r(Q)} \sum_j \left( \frac{1}{|Q_j'|}\int_{Q_j'} w^r \right)^\frac{p'}r\int_{Q_j}w^{p'-1}\right)^\frac{1}{p'} \\
& \leq 2^n L \left( \frac{1}{w_r(Q)} \sum_j  w_r(Q_j') \left( \frac{1}{|Q'_j|}\int_{Q_j'} w^r \right)^\frac{p'-1}{r}\left( \frac{1}{|Q_j'|}\int_{Q_j}w^{p'-1}\right)^{(p'-1)(p-1)}  \right)^\frac{1}{p'} \\
& \leq 2^n L [w]_{A^r_p}^\frac 1p \left(\frac{1}{w_r(Q)}\sum_j |Q_j'| \left( \frac{1}{|Q_j'|} \int_{Q_j'}w^r\right)^\frac1r \right) ^ \frac{1}{p'}.
\end{align*}

In order to bound the term in the sum, we recall the following result by Kolmogorov. If $(X,\mu)$ is a probability space, then for $\eps<1$
\[ \| g\|_{L^{\eps}(X)} \leq \left( \frac{1}{1-\eps} \right)^\frac 1\eps \, \|g\|_{L^{1,\infty}(X)} .\]
We have 
\begin{align*}
\sum_j |Q_j'| \left( \frac{1}{|Q_j'|} \int_{Q_j'}w^r\right)^\frac1r & 
\leq 2^n  \sum_j |Q_j| \inf_{z\in Q_j} M_Q(w^r\car Q)(z)^\frac 1r\\
& \leq 2^n |Q| \avgint_Q M_Q(w^r \car Q)^\frac 1r \\
& \leq 2^n \frac{1}{1-\frac1r} \| M_Q(w^r\car Q)\|_{L^{1,\infty}(Q, \frac{dx}{|Q|})}^\frac 1r |Q| \\
& \leq 2^n r' |Q| \left( \avgint_Q w^r\right)^\frac 1r \\
& = 2^n \, r' \, w_r(Q),
\end{align*}
where $M_Q$ is as before the local dyadic maximal operator over $Q$ whose weak type $(1,1)$ bound is one.
Thus, we have the bound
$$ A_2 \leq 2^n \,[w]_{A^r_p}^\frac 1p (r')^\frac{1}{p'} \, L.$$

Combining the bounds for $A_1$, $A_2$ and $B,$ we have for every cube $Q$ and $L>1$
\begin{equation}\label{KeyEstimate}
\left( \frac{1}{ w_r(Q) } \int_{ Q }   \left(\frac{|f -f_{Q}|}{w}\right)^{p'}     \,wdx\right)^{\frac{1}{ p'} } 
\leq L+2^n \,[w]_{A^r_p}^\frac 1p (r')^\frac{1}{p'} \, L+ X \left( \frac1L \right)^\frac{1}{p'} 
\end{equation}
and thus for each L
$$
X \leq 2^{n+1} \,[w]_{A^r_p}^\frac 1p (r')^\frac{1}{p'} \, L+ X \left( \frac1L \right)^\frac{1}{p'}. 
$$
Hence, if we assume $X<\infty$, 
$$
X \leq c_n \,p'[w]_{A^r_p}^\frac 1p (r')^\frac{1}{p'},
$$
This finishes the proof in the case that $X<\infty$. In order to remove the hypothesis $X<\infty$, it is enough to replace first for each cube $Q$
$$
\left( \frac{1}{ w_r(Q) } \int_{ Q }   \left(\frac{|f -f_{Q}|}{w}\right)^{p'}     \,wdx\right)^{\frac{1}{ p'} } 
$$
by 
$$
\left( \frac{1}{w_r(Q)}\int_Q \min \Big\{ \frac{|f -f_{Q}|}{w},m \Big\}^{p'} wdx\right)^{\frac{1}{p'}}.
$$
The argument done before works exactly to get the following variant of \eqref{KeyEstimate}: For every $L>1$ and $m=1, \cdots,$
$$
\left( \frac{1}{w_r(Q)}\int_Q \min \Big\{ \frac{|f -f_{Q}|}{w},m \Big\}^{p'} wdx\right)^{\frac{1}{p'}}\leq  L+2^n \,[w]_{A^r_p}^\frac 1p (r')^\frac{1}{p'} \, L+ X_m \left( \frac1L \right)^\frac{1}{p'},% \qquad m=1,\cdots, L>1,
$$
where now, instead of $X$ we have $X_m$ defined by: 
$$X_{m} := \sup_{Q\in \mathcal D} \left( \frac{1}{w_r(Q)}\int_Q \min \Big\{ \frac{|f -f_{Q}|}{w},m \Big\}^{p'} wdx\right)^{\frac{1}{p'}} \qquad m=1,\cdots.$$
Then,
$$
X_m \leq 2^{n+1} \,[w]_{A^r_p}^\frac 1p (r')^\frac{1}{p'} \, L+ X_m  \left( \frac1L \right)^\frac{1}{p'} \qquad L>1, \: m=1,\cdots .
$$
Therefore, since  $X_{m}\leq  m$ we have 
$$
X_m \leq c_n \,p'[w]_{A^r_p}^\frac 1p (r')^\frac{1}{p'} \qquad m=1,\cdots.  
$$
Hence for each cube $Q$
$$
\left( \frac{1}{w_r(Q)}\int_Q \min \Big\{ \frac{|f -f_{Q}|}{w},m \Big\}^{p'} wdx\right)^{\frac{1}{p'}}\leq 
c_n \,p'[w]_{A^r_p}^\frac 1p (r')^\frac{1}{p'} 
\qquad m=1,\cdots. 
$$
Finally, let $m\rightarrow \infty$ to finish the proof.
\end{proof}

\begin{proof}[Proof of Corollary \ref{ThmBloomBMO}]
Part (1) follows from part (2) since $[w]_{A_1} \geq [w]_{A_p},$ $p>1$. In order to prove part (2), choose $r= 1+\delta$ with $\delta$ as in Theorem \ref{SHARP-RHI-Ainfty}. This way $[w]_{A_p^r} \leq 2[w]_{A_p}$, $r' \lesssim [w]_{A_\infty}$ and $w_r(Q) \leq 2w(Q)$. The result follows from Theorem \ref{bumped-muckenhoupt-wheeden}.
\end{proof}

\section{Application to $C_p$ weights} \label{Cp}

In this section we will prove Theorem \ref{Yabuta-quantif}, namely we give a quantitative weighted norm inequality between the Hardy--Littlewood maximal operator and the Fefferman--Stein maximal function, for weights in class $C_p$. We are going to use the improved John--Nirenberg Theorem \ref{Teorema1} to give a quantitative version of Theorem II in \cite{YCF}.

Before proving Theorem \ref{Yabuta-quantif}, we need to obtain a non-dyadic unweighted version of Corollary~\ref{COROLARIO}.

\begin{theorem}
\label{Teoremanodiadico}
Let $Q$ be an arbitrary cube and $f$ a locally integrable function, non constant on $Q$. Then for any $\la>0$ we have
$$\left| \left\lbrace x \in Q: \frac{M\big( (f-f_Q)\car Q\big)(x)}{M^\sharp f(x)} > \la \right\rbrace \right| \leq C e^{-c\la} |Q|,$$
where $C,c>0$ are dimensional constants. Here $M$ denotes the standard Hardy--Littlewood maximal operator.
\end{theorem}

We will use a result from \cite{CONDE}, which will allow us to obtain the general case from the dyadic setting.

\begin{lemma}\label{Lema-cubo-diadico}
Let $Q\subset \R^n$ be a cube. Then there exist $n+1$ dyadic systems $\{\mathcal A_j\}_{j=0}^n$ and $n+1$ cubes, $Q_j \in \mathcal A_j$ such that the following two conditions are satisfied
\begin{enumerate}
\item $Mf(x) \leq c_n \sum_{j=0}^n M_j f(x)$ a.e. for any function $f$, where $M_j$ is the dyadic maximal function with respect to the dyadic system $\mathcal A_j$, $j=0,...,n$.

\item $ Q\subset \cap_{j=0}^n Q_j$ and the $|Q| \simeq |Q_j|$ for all $j$.
\end{enumerate}
\end{lemma}

\begin{proof}
Given the cube $Q$, we construct the dyadic systems as in Theorem A in \cite{CONDE}, but with a slight change on the starting cubes. 

We choose the cubes $Q_{00}^j$ so that $Q = \cap_j Q_{00}^j$. This is possible by construction, after making a translation and dilation. 
Indeed, we may suppose $Q=[\frac{p_n-1}{p_n},1]^n$, $p_n$ being the smallest odd integer strictly greater than $n$. Following  the proof in \cite{CONDE}, we have $Q_{00}^j=[0,1]^n+\frac{j}{p_n}(1,..,1)$. 
These cubes satisfy property $(2)$. Then call $Q_j = Q^j_{00}$ and apply the same procedure as in \cite{CONDE}.
\end{proof}

\begin{proof}[Proof of Theorem \ref{Teoremanodiadico}]
Fix the cube $Q$ and the function $f$, and choose $Q_0,...,Q_n$ and $\mathcal A_0,...,\mathcal A_n$ as in Lemma \ref{Lema-cubo-diadico}. We have
\begin{align*}
\left| \left\lbrace x \in Q: \frac{M\big( (f-f_Q)\car Q\big)}{M^\sharp f} > \la \right\rbrace \right| & 
\leq \sum_{j=0}^n \left| \left\lbrace x \in Q: \frac{M_j\big( (f-f_Q)\car Q\big)}{M^\sharp f} > \frac{\la}{n+1} \right\rbrace \right| \\
& \leq \sum_{j=0}^n \left| \left\lbrace x \in Q_j: \frac{M_j\big( (f-f_Q)\car {Q_j}\big)}{M^\sharp f} > \frac{\la}{n+1} \right\rbrace \right| \\
&= \sum_{j=0}^n \left| \left\lbrace x \in Q_j: \frac{M_{Q_j}(f-f_Q)}{M^\sharp f} > \frac{\la}{n+1} \right\rbrace \right|.
\end{align*}

Now, since $Q$ and $Q_j$ have comparable size, we have for $x\in Q_j$,
\begin{align*}
\frac{|f_Q-f_{Q_j}|}{M^\sharp f(x)} \leq \frac{|Q_j|}{|Q|}\frac{\avgint_{Q_j}|f-f_{Q_j}|}{\avgint_{Q_j}|f-f_{Q_j}|} \leq c_n.
\end{align*}
So, for $\frac{\la}{n+1} \geq c_n$ we get for each $j$,
\begin{align*}
\left| \left\lbrace x \in Q_j: \frac{M_{Q_j}(f-f_Q)}{M^\sharp f} > \frac{\la}{n+1} \right\rbrace \right| & 
\leq \left| \left\lbrace x \in Q_j: \frac{M_{Q_j}(f-f_{Q_j})}{M^\sharp f} + c_n > \frac{\la}{n+1} \right\rbrace \right|\\
&\leq \left| \left\lbrace x \in Q_j: \frac{M_{Q_j}(f-f_{Q_j})}{M^\sharp f} > \frac{\la}{n+1}-c_n \right\rbrace \right| \\ 
& \leq C e^{-c (\la-c_n) } |Q_j| \leq C e^{-c (\frac{\la}{n+1}-c_n) } |Q|.
\end{align*}
This finishes the proof for $\frac{\la}{n+1}>c_n$. The other case follows since in that case $e^{-\la}$ is bounded from bellow.
\end{proof}

We now give the key estimate, which is a good-$\la$ estimate between $M$ and $M^\sharp$ with exponential decay.

\begin{proposition}\label{good-lambda-exp}
Let $f$ be a function and $\la>0$. Let $\Omega_\la =\{ Mf(x) >\la \}= \cup Q$ as in the Whitney decomposition. Then for any $Q$ in the decomposition and $\gamma$ small enough,
$$|\{x\in Q : Mf(x)>4^n \la , M^\sharp f(x) \leq \gamma \la \}| \leq C e^{-\frac{c}{\gamma}} |Q|.$$ 
\end{proposition}
\begin{proof}
Let $\overline Q$ be the multiple of $Q$ such that $\overline Q \cap (\Omega_\la)^c\not = \emptyset$, as in the Whitney decomposition. We prove that if $x\in Q$ satisfies  $Mf(x)>4^n\la$ and $M^\sharp f(x) \leq \gamma \la$ then 
\begin{equation}\label{estimation}
\frac{M((f-f_{\overline Q})\car{\overline Q})(x)}{M^\sharp f(x)}>\frac 1\gamma.
\end{equation}
Then we can directly apply Theorem \ref{Teoremanodiadico} and we  will be done.
Because of the Whitney decomposition, $Mf(x)>4^n \la $ implies $M(f\car{\overline Q})(x)>4^n\la$. Also as a consequence of the Whitney decomposition, $|f|_{\overline Q} \leq \la,$ so
$$4^n\la \leq M(f\car{\overline Q})(x) \leq M((f-f_{\overline Q})\car{\overline Q})+ |f|_{\overline Q} \leq M((f-f_{\overline Q})\car{\overline Q})+\la,$$
which implies $M((f-f_{\overline Q})\car{\overline Q})(x) >\la$. This proves \eqref{estimation}. Therefore we have
\begin{align*}
|\{x\in Q : Mf(x)>4^n \la , &M^\sharp f(x) \leq \gamma \la \}| \\
 &\leq |\{x\in Q : M((f-f_{\overline Q})\car{\overline Q})(x)>4^n \la , M^\sharp f(x) \leq \gamma \la \}| \\
& \leq \left| \left\lbrace x \in \overline Q : \frac{M((f-f_{\overline Q})\car{\overline Q})(x)}{M^\sharp f(x)}> \frac 1\gamma\right\rbrace \right|\\
& \leq c e^{-\frac{1}{c \gamma}} |\overline Q|
\end{align*}
This ends the proof, since $Q$ and $\overline Q$ have comparable size.
\end{proof}

Now we prove Theorem \ref{Yabuta-quantif}. The proof follows mainly the one in \cite{YCF}, but we use the good-$\la$ inequality from Proposition \ref{good-lambda-exp}. We also keep an eye for the dependence on the constant of the weight, which is in fact our main objective.

\begin{proof}[Proof of Theorem \ref{Yabuta-quantif}]
We may assume, arguing as in \cite{YCF}, that both norms are finite. Define $\Omega _k = \{ x \in \R^n: Mf(x) > 2^k\}$ for $k\in \Z$. We write, following the Whitney decomposition theorem
$$\Omega_k = \bigcup_j Q_j^k, \quad Q_j^k \ \text{disjoint cubes}.$$
By Proposition \ref{good-lambda-exp} we have
$$ |\{x\in Q_j^k : Mf(x) > 4^n 2^k , M^\sharp f(x) \leq \gamma \la\}| \leq C e^{-\frac c\gamma}|Q_j^k|,$$
which in turn yields, using Theorem \ref{SHARP-RHI-Cp},
$$ w(\{x\in Q_j^k : Mf(x) > 4^n 2^k , M^\sharp f(x) \leq \gamma \la\}) \leq C e^{-c\frac{\eps}{\gamma}}\int_{\R^n} (M\car{Q_j^k})^q w, $$
where $\eps = \frac{c_n}{ \max(1,[w]_{C_q})}.$ These computations, together with the standard argument that uses the good-$\la$ technique yield
\begin{align*}
\int Mf^p w &
 \leq 2^p  \sum_{k\in \Z} 2^{kp} w(\Omega_k) \\
& \leq (c_n)^p \sum_{k\in \Z} 2^{kp} w(M^\sharp f > \gamma 2^k) + c_n e^{-\frac{c\eps}{\gamma}} \sum_{k,j} 2^{kp} \int_{\R^n} (M\car{Q_j^k})^q w \\
& \leq  \left(\frac {c_n}{\gamma}\right)^p \int M^\sharp f^p w + c_n e^{-\frac{c\eps}{\gamma}} \int (M_{p,q} Mf)^p w,
\end{align*}
where $M_{p,q}$ is the Marcinkiewicz operator as in \cite{CCF}, \cite{SCF}. We now use Lemma 5.8 from \cite{CCF} and obtain
$$ \int  (M_{p,q} Mf)^p w \leq 2^{c_n \frac{pq}{q-p}}\frac1 \eps \log \frac 1\eps  \int Mf^p w.$$
So, if we choose
$$\frac1\gamma = c_{n} \frac{pq}{q-p} \frac1 \eps \log \frac{1}{\eps} = c_{n} \frac{pq}{q-p}  {\max(1, [w]_{C_q}\log^+[w]_{C_q})},$$
we can absorb the last term to the left side and we obtain 
$$ \norm{Mf}_{L^p(w)} \leq c_n \frac{pq}{q-p} \max(1,[w]_{C_q} \log^+[w]_{C_q}) \norm{M^\sharp f}_{L^p(w)}.$$
This finishes the proof.
\end{proof}

Next we provide a proof of Theorem \ref{CpOperator(D)condition}. Recall that we say that an operator $T$ satisfies the \eqref{PROPIEDAD} property if there  are some constants $\delta\in(0,1)$ and $c>0$ such that for all $f$,
\begin{equation}
\label{PROPIEDAD}\tag{D} M^\sharp_\delta (Tf)(x) \leq c Mf(x), \quad a.e. \: x.
\end{equation}
Here $M$ denotes the standard Hardy--Littlewood maximal operator and $M^\sharp _\delta f = M^\sharp(f^\delta)^\frac1	\delta$ 
\begin{proof} [Proof of Theorem \ref{CpOperator(D)condition}]
Since $\frac{p}{\delta}<q$, we can make the following computations:
\begin{align*}
\norm{Tf}_{L^p(w)} & \leq \norm{M_\delta (T f)}_{L^p(w)} 
 = \norm{M(Tf^\delta)}_{L^\frac p\delta (w)}^\frac1\delta \\
& \leq c_{n} \left( \frac{pq}{\delta q-p} \max(1,[w]_{C_q}\log^+[w]_{C_q})\right)^\frac1\delta \norm{M^\sharp(Tf^\delta)}_{L^\frac p\delta (w)}^\frac1\delta \\
& = c_{n} \left( \frac{pq}{\delta q-p}\max(1,[w]_{C_q}\log^+[w]_{C_q}) \right)^\frac1\delta\norm{M_\delta^\sharp(Tf)}_{L^p (w)} \\
& \leq c_{n} \ C_T \ \left(\frac{pq}{\delta q-p}  \ \max(1,[w]_{C_q}\log^+[w]_{C_q}) \right)^\frac1\delta\norm{Mf}_{L^p (w)}.\qedhere
\end{align*}
\end{proof}

\section{Generalized Poincar\'e inequalities: Proof of Theorem \ref{Automejora}}\label{GeneralizedPI}

 In this section we provide a proof of Theorem \ref{Automejora}.   Let $w \in A_\infty$ and let $a$ be a functional satisfying the $D_r(w)$ condition for some $r>1$. More precisely, for every cube $Q$ and every collection $\Lambda$ of pairwise disjoint subcubes of $Q$, the following inequality has to hold:
\begin{equation*}
\sum_{P\in \Lambda} w(P) \, a(P)^r \leq \|a\|^r w(Q) a(Q)^r.
\end{equation*}

This kind of functionals were studied in relation with self improvement properties of Poincar\'e and Poincar\'e--Sobolev inequalities in \cite{PR-Poincare}. Here we present a result similar to the ones obtained there.

\begin{proof} [Proof of Theorem \ref{Automejora}]   
Fix a cube $Q$. We have to prove that for every $t>0$,
\begin{equation}
t^r \, w(\{x\in Q: |f(x)-f_Q|>t\})\leq (c_n \|a\| [w]_{A_\infty})^r \, a(Q)^r \, w(Q),
\end{equation}
with $c_n$ independent from everything but the dimension.

 $M_Q$ will denote the dyadic maximal operator localized in $Q$. Since $(f-f_Q) \leq M_Q (f-f_Q),$ we can just estimate the bigger set $\Omega_t = \{x\in Q: M_Q(f-f_Q)(x)>t\}$. Let $Q_j$ be the maximal cubes that form $\Omega_t$. They can be found by  the Calder\'on--Zygmund decomposition. Let $q=2^n+1$ as in \cite{PR-Poincare}, and let us make the same computations that they do. We arrive to
$$w(\Omega_{qt}) \leq \sum_j w(E_{Q_j}),$$
where 
\begin{align*}
E_{Q_j} &= \{ x\in Q_j: M_Q(f-f_{Q_j})(x)>t\} \\
&= \{ x\in Q_j: M_{Q_j}(f-f_{Q_j})(x)>t\} ,
\end{align*}
by the maximality of the cubes $Q_j$. Now we will use the good-$\la$ from Corollary \ref{COROLARIO}. We use the version with the dyadic sharp maximal function in Remark \ref{REMARK-DIADICO}. Let $\gamma>0$ to be chosen later. Then
\begin{align*}
E_{Q_j}&  \subseteq \{M_{Q_j}(f-f_{Q_j})>t, \, M_d^\sharp f \leq \gamma t\}  \cup \{  M_d^\sharp f > \gamma t\}  = A_j \cup B_j
\end{align*}
and therefore
\begin{align*}
w(E_{Q_j} )  &\leq w(A_j) + w(B_j).
\end{align*}

For $A_j$-s, let $s>1$ be the exponent for the Reverse H\"older inequality for $w\in A_\infty$ as in Theorem \ref{SHARP-RHI-Ainfty}. Then, using Corollary \ref{COROLARIO}, we have
\begin{align*}
\sum_j w(A_j) \leq c_1 e^{-\frac{c_2}{ s \gamma} } \sum_j w(Q_j) = c_1 e^{-\frac{c_2}{ s \gamma} } w(\Omega_t).
\end{align*}
Remember that $c_1,c_2>0$ are dimensional constants. On the other hand, for $B_j$ we can argue as follows. We have
\begin{align*}
\bigcup_j B_j \subseteq \{x\in Q: M_d^\sharp f(x)>\gamma t\}=\bigcup_i R_i,
\end{align*}
where $R_i$ are the maximal dyadic subcubes of $Q$ such that 
$$
 \gamma t < \frac{1}{|R_i|}\int_{R_i}|f-f_{R_i}| .
$$
Now, using the starting point \eqref{starting-point}, we clearly have 
$$ \gamma t \leq a(R_i).$$
Therefore, using that $a$ satisfies the $D_r(w)$ condition, we have
\begin{align*}
\sum_j w(B_j) & \leq w\big( \{x\in Q: M_d^\sharp f(x)>\gamma t\} \big) \\
& = \sum_i w(R_i)\\
& \leq \left( \frac{1}{\gamma t}\right)^r\sum_i w(R_i) a(R_i)^r \\
& \leq \|a\|^r \left( \frac{1}{\gamma t}\right)^r w(Q) a(Q)^r.
\end{align*}

Now, if we put everything together, we get
\begin{align*}
(qt)^r w(\Omega_{qt}) & \leq c_1 (tq)^r e^{-\frac{c_2}{\gamma s}}w(\Omega_t) + \left(q  \frac{\|a\|}{\gamma}\right)^r  w(Q) a(Q)^r.
\end{align*}
Since we have $qt$ on the left and $t$ on the right, we define the function \[ \varphi(N) = \sup_{0<t\leq N} t^r w(\Omega_t).\] This function is increasing, so we have
$$\varphi(N) \leq \varphi(Nq) \leq c_1 q^r e^{-\frac{c_2}{\gamma s}}\varphi(N) + \left(q  \frac{\|a\|}{\gamma}\right)^r  w(Q) a(Q)^r.$$
The parameter $\gamma$ is free, and we make the choice so that 
$$ c_1 q^r e^{-\frac{c_2}{s\gamma}}=\frac 12,$$
which means $$\gamma = \frac{c_n}{r [w]_{A_\infty}}.$$
This yields the result, since $\|f-f_Q \|_{L^{r,\infty}(Q,w)}\leq \sup_N\varphi(N)$.
\end{proof}

\section{Further extensions: polynomial approximation}\label{polynomial}

In this section we generalize Theorems \ref{Teorema1} and  \ref{bumped-muckenhoupt-wheeden} to the context of polynomials. More precisely, we show that the average $f_Q$ can be replaced with an appropriate polynomial $P_Q f$ of fixed degree $k$. It is not clear how to obtain this polynomial approximation from the sparse techniques in \cite{OPRRR19}.

Let $\mathcal P_k(Q)$ denote the space of polynomials of degree at most $k$ restricted to the cube $Q$, and let $m_k$ denote the dimension of $\mathcal P_k(Q)$, which depends only on $k$ and $n$. The degree $k$ will be frozen from now on, so we omit the subscript $k$ if there is no room for confusion.

We are going to work with the $L^2(Q)$ space with normalized measure, namely we consider the standard product $\langle f,g\rangle_Q = \avgint_Q f\overline g$. First, we have to construct an orthonormal basis of $\mathcal P(Q)$. We choose any orthonormal basis of $\mathcal P( [0,1]^n)$, namely $\{ e_j\}_{j=1}^{m_k}$. For a general cube $Q = y + \ell [0,1]^n$, we choose the basis formed by 
\[
e_{j,Q}(x) = e_j\big( \frac{x-y}{\ell} \big).
\]
Note that $\{e_{j,Q}\}_j$ is indeed an orthonormal basis because the measure in $Q$ is normalized. Moreover, for a fixed degree $k$, all the basis vectors have uniformly bounded $L^\infty$ norm for every cube $Q.$ If we were to increase the degree $k$, we would just have to introduce new vectors to the basis. We define the orthogonal projection operator
\[ \begin{array}{r c l}
P_Q: L^1(Q)  & \longrightarrow & \mathcal P(Q) \\
f & \longmapsto & \sum_{j=0}^{m_k} \langle f , e_{j,Q} \rangle_Q \ e_{j,Q} .
\end{array} 
\]

Notice then that that the projection operator is indeed defined in the whole $L^1(Q)$ because the $e_{j,Q}$ are polynomials and therefore they belong to $L^\infty$. Using the fact that the vectors $\{e_{j,Q}\}$ are uniformly bounded, one can prove
\begin{equation}\label{property-boundedness}
|P_Qf(x)| \leq \gamma \ \avgint_Q |f|,
\end{equation}
for any $f\in L^1(Q)$, and where $\gamma$ is a constant depending only on the dimension $n$ and on~$k$. 

Combining these properties we can show the following optimality property, when $p\in[1,\infty)$:
\begin{equation*}
\inf_{\pi\in\mathcal{P}_k}\left (\avgint_Q |f-\pi|^p\right) ^{1/p} \approx \left (\avgint_Q|f-P_Qf|^p \right )^{1/p}.
\end{equation*}
Indeed, the inequality in the direction ``$\le$'' is  
trivial. To prove the opposite inequality, observe that since $P_{Q}$
is a projection we have  
$P_{Q}\pi =\pi$ for any polynomial of degree at most $k$, and
therefore  by the triangle inequality
\begin{equation*}
\left (\avgint_Q|f-P_Qf|^p \right )^{1/p}\leq   \left (\avgint_Q|f-\pi |^p \right )^{1/p}+ \left (\avgint_Q| P_Q(f-\pi)|^p \right )^{1/p}
\end{equation*}
\begin{equation*}
\leq   (1+\gamma) \left (\avgint_Q |f-\pi|^p \right )^{1/p}
\end{equation*}
by \eqref{property-boundedness}.

Before we state the main results of this section, we introduce the sharp maximal function in this polynomial  context, which has the expected form
\[M^\sharp_kf(x) = \sup_{Q \ni x} \frac{1}{|Q|} \int_Q |f-P_Q f|\]
where $k$ is the degree and where $P_Q f$ is the optimal polynomial mentioned above of degree less than $k$. The case $k=0$ corresponds to the usual sharp maximal function.
We first state the maximal polynomial theorem:

\begin{theorem}\label{Teorema1withPolynomials}
Let $f\in L^1_{loc}$, $Q$ a cube, $1<r<\infty$ and $1\leq p<\infty$. Then
\[
 \left( \frac{1}{w_r(Q)} \int_Q \Big( \frac{M_Q(f-P_Qf)(x)}{M_k^\sharp f(x)}\Big)^p w(x) dx \right)^\frac{1}{p} \leq c_n \ r' \  \gamma \ p.
 \]
\end{theorem}

We now state the polynomial version of Theorem \ref{bumped-muckenhoupt-wheeden}. We introduce the weighted polynomial $\BMO$ norm, that is, for a certain weight $w$ we define 
\[ \|f \|_{\BMO_k^r(w)}:= \sup_{Q} \frac{1}{w_r(Q)} \int_Q |f-P_Qf|  . \]

\begin{theorem}\label{ThmBloomBMOwithPolynomials} 
Let $1<p<\infty$ and $r> 1$. Let $w$ a weight and $f$ a function satisfying $[w]_{A_p^r}<\infty$ and $\|f\|_{\BMO_k^r(w)} <\infty$. Then  
\[
\left(\frac{1}{w_r(Q)}\int_{Q}\left|\frac{f(x)-P_Qf}{w(x)}\right|^{p'}w(x)dx\right)^{\frac{1}{p'}}\leq c_{n}\, \gamma \, p'\, \big([w]_{A^r_{p}}\big)^{\frac{1}{p}}\,\big(r'\big)^{\frac{1}{p'}}\,\|f\|_{\BMO_{k}^r(w)} .
\]
\end{theorem}

\begin{remark}
One can also obtain $A_\infty$ results analogous to Corollaries \ref{COROLARIO} and \ref{ThmBloomBMO}.
\end{remark}

Since the proofs of these theorems are very similar to the zero degree case but making only the appropriate changes, we are only going to give the proof of Theorem \ref{Teorema1withPolynomials}, and just a sketch of the proof of Theorem \ref{ThmBloomBMOwithPolynomials}.

\begin{proof}[Proof of Theorem \ref{Teorema1withPolynomials}]
 Fix $L>1$ and make the Calder\'on--Zygmund decomposition of the function 
\[ F(x) =  \frac{|f(x)-P_Qf(x)|}{\osc_k(f,Q)},\]

where now \[ \osc_k(f,Q)= \avgint_Q|f-P_Qf|.\]

We obtain cubes $\{Q_j\}$ that satisfy:
\begin{itemize}
\item $ \displaystyle L < \avgint_{Q_j} \frac{|f(x)-P_Qf(x)|}{\osc_k(f,Q)} \leq 2^n L$
\item $\displaystyle \sum_{Q_j} |{Q_j}| \leq \frac{|Q|}{L}$
\item for any $x\not \in \bigcup {Q_j},$ it holds $\displaystyle \frac{|f(x)-P_Qf(x)|}{\osc_k(f,Q)} \leq L$
\end{itemize}

Fix one of these cubes ${Q_j}$ and let $x\in {Q_j}$. We can localize by maximality, meaning:
\begin{align*}
M_Q(f-P_Q)(x) = M_{Q_j}(f-P_Qf)(x) \leq M_{Q_j}(f-P_{Q_j}f)(x) + M_{Q_j}(P_{Q_j}f-P_Qf)(x).
\end{align*}
Now, $P_{Q_j}f-P_Qf$ is not constant, but we can bound it anyway. Indeed, since both are polynomials of degree at most $k$, ${Q_j}\subset Q$ and both $P_{Q_j}$ and $P_Q$ are projection operators, we have 
\[ P_{Q_j}f -P_Qf = P_{Q_j}f -P_{Q_j}(P_Qf) = P_{Q_j}(f-P_Qf)\]
Therefore, using \eqref{property-boundedness} we get
\[ |P_{Q_j}f(x)-P_Qf(x) | \leq | P_{Q_j}(f-P_Q)(x) | \leq \gamma \avgint_{Q_j} |f-P_Qf| \leq 2^n \ L \ \gamma \ \osc_k(f,Q).\]
And, since the maximal function is bounded in $L^\infty$ with norm one, we directly have 
\[ M_{Q_j}(P_{Q_j} f-P_Qf)(x)  \leq \gamma 2^n L\ \osc_k(f,Q). \]
Now, this means that for $x\in {Q_j},$
\[ \frac{M_Q(f-P_Qf)(x)}{M_k^\sharp f(x) } \leq 2^n L \gamma + \frac{M_{Q_j}(f-P_{Q_j}f)(x)}{M_k^\sharp f(x) }\]
and therefore

\begin{align*}
\left( \frac{1}{w_r(Q)} \int_Q \Big( \frac{M_Q(f-P_Qf)}{M_k^\sharp f }\Big)^p w \right)^\frac 1p 
& \leq 2^n L \gamma  \\
& \: \: +   \left( \frac{1}{w_r(Q)}\sum_{j}  \frac{w_r(Q_j)}{w_r(Q_j)} \int_{Q_j} \Big( \frac{M_{Q_j}(f-P_{Q_j}f)}{M^\sharp f } \Big)^p w \right)^\frac 1p \\
& \leq 2^n \gamma L + \frac{X}{L^\frac1p}.
\end{align*}
From here, the result follows as in the proof of Theorem \ref{Teorema1}.
\end{proof}

\begin{proof}[Sketch of the proof of Theorem \ref{ThmBloomBMOwithPolynomials}]
We fix $L>1$ and make the mixed-type Calder\'on--Zygmund decomposition at height $L$ of the function \[ |f-P_Qf|.\] That is, we select the maximal cubes $\{ Q_j \}$ that satisfy
\[
\frac{1}{w_r(Q_j)} \int_{Q_j} |f-P_Qf| >L.
\] 
As in the proof of Theorem \ref{Teorema1withPolynomials}, these cubes will satisfy
\begin{itemize}
\item $\sum w_r(Q_j) \leq \frac{w_r(Q)}{L}$
\item For a.e. $x\in Q_j$, $ |P_Qf(x)-P_{Q_j}f(x)| \leq 2^n \ L \ \gamma \frac{w_r(Q_j')}{|Q_j'|},$ where $Q_j'$ is the parent of $Q_j$
\item $|f(x)-P_Qf(x)| \leq Lw(x)$ almos everywhere outside of $\cup Q_j$
\end{itemize}
Therefore, one can compute as before
\begin{align*}
\left( \frac1{w_r(Q)} \int_Q \Big( \frac{f-P_Qf}{w}\Big)^q w \right)^\frac1q & \leq 
\left( \frac{1}{w(Q)}\int_{(\cup Q_j)^c} L w \right)^\frac1q \\
&\quad  + \left( \frac{1}{w_r(Q)} \sum_j \int_{Q_j} |P_{Q_j}f-P_Qf|^q w^{1-q}\right)^\frac1q \\
& \quad +  \left(\frac{1}{w_r(Q)} \sum_j \int_{Q_j} \Big( \frac{|f-P_{Q_j}f|}{w}\Big)^q w \right)^\frac1q \\
&= A_1+ A_2+ B.
\end{align*}
Clearly, $A_1\leq L$ and $B \leq \frac{X}{L^{\frac{1}{q}}}$, where
\[
X = \sup_{R \in \mathcal D(Q)} \left( \frac1{w_r(R)} \int_R \Big( \frac{f-P_Rf}{w}\Big)^q w \right)^\frac1q .
\]
In order to bound $A_2$ we can argue as in the proof of Theorem \ref{bumped-muckenhoupt-wheeden} but using the new properties of the Calder\'on--Zygmund cubes to get
\[
A_2 \leq 2^n \ \gamma \ L \big(r'\big)^\frac{1}{q} \ \big( [w]_{A_{q'}^r} \big)^{\frac1{q'}}. 
\]
The proof follows as in the proof of Theorem \ref{bumped-muckenhoupt-wheeden}.
\end{proof}

\appendix

\section{The John--Nirenberg theorem}

In this appendix, we give a proof of the John--Nirenberg theorem that seems to have been overlooked in the literature.
We include the proof here for convenience of the reader.

\begin{theorem}
Let $f\in \BMO$ and $Q$ a cube. Then, for some dimensional constant $c_n$,
\begin{equation*}
|\{ x\in Q: |f(x)-f_Q|>t\} | \leq 2\, e^{\textstyle -\frac{c_n\, t}{\|f\|_{\BMO}}} |Q|.
\end{equation*}
\end{theorem}
\begin{proof}
We just need to combine Proposition \ref{proposicion-JN} with the exponential estimate from Proposition \ref{EXPONENTIAL}.
\end{proof}

\begin{proposition}\label{proposicion-JN}
Let $f\in \BMO$. Then for every cube $Q$ and $p\geq1$,
\begin{equation*}\label{basic-self-improving}
\left(\frac{1}{|Q|}\int_Q |f-f_Q| ^p \right)^\frac{1}{p} \leq c_n \, p \, \|f\|_{\BMO}.
\end{equation*}
\end{proposition}

\begin{remark}
Usually this result is presented ad a corollary of the JN theorem.
\end{remark}

\begin{proof}[Proof of Proposition \ref{proposicion-JN}] We may suppose $\|f\|_{\BMO} =1$ by homogeneity. 
Let $L>1$ to be chosen. We make the Calder\'on--Zygmund decomposition of $|f-f_Q|$ in $Q$ at height $L$. We obtain a family $\{Q_j\}$ of dyadic subcubes of $Q$. These cubes are pairwise disjoint with respect to the property $$L < \frac{1}{|Q_j|}\int_{Q_j} |f-f_Q|\leq 2^n L .$$
Moreover, if $x\not \in \cup_j Q_j$, then $|f(x)-f_Q|\leq L$.

Using the disjointness, we have for $a.e. \ x\in Q$,
\begin{align*}
f(x)-f_Q & = (f(x)-f_Q) \car {(\cup_j Q_j)^c} (x) + \sum_{j} (f_{Q_j}-f_Q) \car{Q_j}(x) + \sum_j (f(x)-f_{Q_j})\car{Q_j}(x) \\
&= A_1(x) + A_2(x) + B(x).
\end{align*}

%We now use the triangular inequality to get
%%
%\begin{align*}
%\left(\frac{1}{|Q|}\int_Q |f-f_Q|^p \right)^\frac{1}{p} & \leq \left(\frac{1}{|Q|}\int_Q A_1 ^p \right)^\frac{1}{p} + \left(\frac{1}{|Q|}\int_Q A_2 ^p \right)^\frac{1}{p}+ \left(\frac{1}{|Q|}\int_Q B ^p \right)^\frac{1}{p}
%\end{align*}

By the Calder\'on--Zygmund decomposition, we have $|A_1| \leq L$ and $|A_2|\leq 2^n L$ a.e., so $|A_1+A_2| \leq 2^n L$ since they have disjoint support.
 Now, for the remaining part, we compute the norm
\begin{align*}
\left( \frac{1}{|Q|} \int_Q |B(x)|^p\right)^\frac 1p & = \left(\frac{1}{|Q|}\sum_j \int_{Q_j} |f(x)-f_{Q_j}|^p \right)^\frac{1}{p} \\
& \leq \sup_{R } \left( \frac{1}{|R|} \int_R |f-f_R|^p\right)^\frac1p \left( \sum_j \frac{|Q|}{|Q_j|}\right)^\frac{1}{p} \\
&\leq \frac{X}{L^\frac{1}{p}},
\end{align*}
where $X$ equals the corresponding supremum, which is taken over all cubes $R$.

Combining the estimates for $A_1,A_2$ and $B$, we have
\begin{equation}\label{Ultima-John-Nirenberg}\left( \frac{1}{|Q|} \int_Q |f(x)-f_Q|^p\right)^\frac 1p \leq 2^n L + \frac{X}{L^\frac{1}{p}}.
\end{equation}
Since \eqref{Ultima-John-Nirenberg} holds for all cubes $Q$ and $L>1$, the right hand side being independent from $Q$, we can take the supremum over all cubes $Q$ and we get
$$X \leq 2^n L + \frac{X}{L^\frac{1}{p}}.$$
Passing the last term to the left, we get 
$$X \leq 2^n \inf_{L>1} L (L^\frac{1}{p})' \leq 2^{n+1} \,e \,p,$$
where in the last inequality we used Lemma \ref{minimization}. But we can only do this if $X<\infty$, which a priori could not be true. Of course, one can use the John--Nirenberg theorem to prove $X \simeq_p \|f\|_{\BMO}$, but since we are providing an alternative proof of John--Nirenberg we cannot do that.

The way of avoiding this problem is by making a truncation of $f$ at height $m>0$. If $f$ is a real function, call $f_m$ the truncation
\begin{equation*}%\label{f-truncada}
f(x) = \begin{cases}
-m, & f(x)<-m\\
f(x), & -m\leq f(x)\leq m \\
m, & f(x)>m.
\end{cases}
\end{equation*} 
In the case that $f$ is complex, one can do a similar trick. In any case, it is easy to prove
$$
\frac{1}{|Q|}\int_Q |(f_m)-(f_m)_Q| \leq \frac{2}{|Q|}\int_Q |f-f_Q|
$$
If we work with the functions $f_m$ instead of $f$, arguing in the same way
$$ X_m \leq 2^n L + \frac{X_m}{L^\frac{1}{p}}.$$
But now, $X_m \leq 2m<\infty$, so the rest of the proof can continue. The last step is to let $m\rightarrow  \infty$ with the help of Monotone Convergence.
\end{proof}

\bibliographystyle{amsalpha}

\end{document}